\documentclass[options]{amsart}

\usepackage{latexsym,ifthen,amssymb}
\usepackage[toc,page,title,titletoc,header]{appendix}
\usepackage{enumerate}
\usepackage{graphicx}
\usepackage{bm}
\usepackage{subfigure}
\usepackage{float}
\usepackage{rotating}
\usepackage{epstopdf}
\usepackage{hyperref}

\usepackage[toc,page,title,titletoc,header]{appendix}
\usepackage{multirow}
 \usepackage{longtable}
 \usepackage{rotating}

\usepackage{url}
\usepackage{xcolor}	
\usepackage{soul}

\newtheorem{theorem}{Theorem}[section]
\newtheorem{lemma}[theorem]{Lemma}

\newtheorem{claim}[theorem]{Claim}
\newtheorem{corollary}[theorem]{Corollary}
\newtheorem{proposition}[theorem]{Proposition}

\theoremstyle{definition}
\newtheorem{definition}[theorem]{Definition}

\theoremstyle{remark}
\newtheorem{remark}[theorem]{Remark}
\newtheorem {question}[theorem]{Question}
\numberwithin{equation}{section}

\newtheoremstyle{noparens}%
  {}{}%
{}{}%
{\bfseries}{.}%
{ }%
{\thmname{#1}\thmnumber{ #2}\thmnote{ #3}}
\theoremstyle{noparens}
\newtheorem*{question*}{Question}
\newtheorem*{theorem*}{Theorem}

\newcommand{\uhr}{{\upharpoonright}}

\def\t{\tilde}

\def\mcal{\mathcal}
\def\mb{\mathbf}

\def\<{<^*}
\def\>{>^*}

\def\h{\hat}
\def\v{\vec}

\title{Avoid Schnorr randomness}

\author{Lu Liu}
\address{Department of Mathematics,
Central South University,
City Changsha, Hunan Province,
China. 410083}
\email{g.jiayi.liu@gmail.com}

\begin{document}

\def\m{m}
\def\defeq{\overset{def}{=}}
\def\mbP{\mathbb{P}}
\def\U{U}
\def\Markov{Markov}
\def\mbE{\mathbb{E}}
\def\Tr{Tr}
\def\heighthomogeneous{height-homogeneous}
\def\allowsplit{allow split}
\def\DNR{\textrm{DNR}}
\def\ch{\check}
\maketitle

\begin{abstract}
We prove that every finite partition of $\omega$
admit an infinite subset that does not compute a Schnorr random real.
We use this result to answer two questions of Brendle, Brooke-Taylor, Ng and Nies
and strength a result of Khan and Miller.
\end{abstract}

\section{Introduction}

Cardinal characteristic study
has been an important direction in set theory.
The recent study of
Brendle, Brooke-Taylor, Ng and Nies\cite{brendle2015analogy}
pointed out an analog between
many results of cardinal characteristic and
results in computability theory.
We answer three questions in their paper
concerning whether it is possible
 to avoid Schnorr randomness
in \DNR.
We consider the question that whether it is possible to
avoid Schnorr randomness in an arbitrary partition
and give a yes answer. Using this result, we
answer two questions in \cite{brendle2015analogy}.
Hanssen \cite{Kjos-Hanssen2009Infinite} showed that
 for every finite partition of $\omega$,
 there exists an infinite subset
 that does not compute any $1$-random.
  \cite{kjos2019extracting} proved that this can be strengthened
 to avoid effective positive hausdorff dimension.
 But Schnorr randomness is essentially different
 in that there does not exists
 countably many computable trees so that
  every Schnorr random real is
  a path on one of them. Therefore it is not known whether
  these results  can be improved to avoid Schnorr randomness.
  Khan and Miller \cite{khan2017forcing} proved that for any order function $h$,
  there exists $\DNR_h$ that does not compute any Kurtz random real.
  Their result take advantage of the fact that all $\DNR_h$
  lies on a computable tree. But it is not known whether for any oracle
  $X$, there exists a $\DNR_h^X$ that does not compute any Schnorr random real.
 Some reference on basic knowledge of computability
 randomness theory are \cite{Nies2009Computability}\cite{Downey2010Algorithmic}.
 \def\mbV{\mb{V}}
 We state our main results and how it answers two questions  of \cite{brendle2015analogy}
 in section \ref{secmain}. The proof is given in section \ref{secproof}.
 In section \ref{secanotherques} we answer another question of
 \cite{brendle2015analogy}.

 \subsection{Preliminaries}
 For a measurable set $\mcal{A}\subseteq 2^\omega$,
 let $\m(\mcal{A})$ denote the Lebesgue measure of $\mcal{A}$;
 for $V\subseteq 2^{<\omega}$,
 let $\m(V)$ denote $\m(\cup_{\sigma\in V}[\sigma])$.

  A $k$-partition of $\omega$ is a function
 $f:\omega\rightarrow k$. For
 every infinite string $X\in l^\omega$, we also
 think of $X$ as a function from $\omega$ to $l$,
 so that it make sense to write $X^{-1}(i)$.

 Denote strings  in $2^{<\omega}$ by Greek letters $\rho, \sigma, \tau,\dots$;
 we think of binary strings as sets therefore it makes sense to write
 $\sigma\cup\tau$, $\sigma\subseteq \tau$.
 We adopt the convention that whenever we write $\sigma\subseteq \tau$,
 it implies $|\sigma|\leq |\tau|$.

We say $\sigma$ is extended by $\tau$ (written $\sigma\preceq\tau$ or $\tau\succeq \sigma$) if it is an initial segment of $\tau$.
The symbol $\prec$ is reserved for proper initial segment, including that of an infinite set $X\subseteq \omega$
(upon identifying $X$ with its characteristic function).
By $\vec{\sigma}$ we mean a finite sequence of pairwise incompatible strings $(\sigma_0,\cdots,\sigma_n)$.
   For a tree $T$,
   we write $|\rho|_T$ for the $T$-length of $\rho$, i.e.~$|\rho|_T =n+1$ where $n$ is the number of proper initial segments of $\rho$ in $T$.
   For a string $\rho \in 2^{<\omega}$, we let $[\rho]^{\preceq}=\{\sigma: \sigma\succeq \rho\}$;
similarly, for $S\subseteq 2^{<\omega}$, let $[S]^{\preceq} = \{\sigma: \sigma\succeq \rho\text{ for some }\rho\in S\}$;
for a tree $T$, let $[T]$ denote the set of infinite path on $T$
and let $[\rho]= \{X\in 2^\omega:X\succeq \rho\}$.

\section{ Subset of partition that does not compute Schnorr random real}
\label{secmain}
Let REC denote the class of all computable sets.
\begin{definition}[\cite{brendle2015analogy}]
A set $\mcal{A}\subseteq 2^\omega$ is
$A$-\emph{effectively meager} if there exists
a sequence of uniformly $\Pi_1^{0,A}$ class
$(Q_m:m\in\omega)$ so that each $Q_m$ is nowhere dense
such that $\mcal{A}\subseteq \cup_m Q_m$.
A set $A$ is \emph{weakly meager covering}
if the class REC is $A$-effectively meager.
\end{definition}
\begin{definition}
 A \emph{$A$-schnorr test} is
 a $A$-computable sequence of finite set $V_0,V_1,\cdots
 \subseteq 2^{<\omega}$ (denoted as
 $\mbV$) so that
 $\m( V_n)\leq 2^{-n} $ for all $n$.
 We say $\mbV$ \emph{succeed}
 on $X\in 2^\omega$ if $X\in \bigcap\limits_{n\in\omega}
 \bigcup\limits_{m>n}[V_m]$.
 We say $X$ is \emph{Schnorr random} if
 there does not exist Schnorr test
 succeed on $X$.

 \end{definition}
Firstly, by results in
\cite{brendle2015analogy}\cite{Rupprecht2010}\cite{Kjos-Hanssen2011Kolmogorov},
 weakly meager covering is characterized  as following.
\begin{theorem}\label{schth0}
A set $A$ is weakly meager covering if and only if
it is high or of $\DNR$ degree.
\end{theorem}
The reason we concern partition instead of \DNR\ is
following.

\begin{theorem}[\cite{Hirschfeldt2008strength}]\label{schth1}
For any oracle $A$, there exists a $2$-partition $f:\omega\rightarrow 2$ of
$\omega$ such that  every infinite subset $G$ of $f$ is of $\DNR^A$ degree.
\end{theorem}
Here comes our main result.
\begin{theorem}\label{schth4}
For every $k$-partition $f:\omega\rightarrow k$,
there exists an infinite subset $G$ of $f$ such that
$G$ does not compute any Schnorr random real.
\end{theorem}
Meanwhile, it's trivial to verify that
\begin{proposition}\label{schth3}
There exists a $2$-partition $f:\omega\rightarrow 2$ such that
every infinite subset $G$ of $f$ is of hyperimmune degree.
\end{proposition}
\begin{proof}
Simply make sure that the principal function of
$f^{-1}(0),f^{-1}(1)$ is not dominated by any computable function.
\end{proof}

Now we can answer Question 4.1-(6) of \cite{brendle2015analogy}.
\begin{corollary}
There exists a set $G$ such that
\begin{enumerate}
\item $G$ is weakly meager covering;

\item $G$ does not compute any Schnorr random real;
\item $G$ is of hyperimmune degree.
\end{enumerate}
\end{corollary}
\begin{proof}

Let $f_0, f_1$ be as in Theorem \ref{schth1},
Proposition
\ref{schth3} respectively.
Let $G$
 be an infinite subset of both $f_0,f_1$
 as in Theorem \ref{schth4} so that it does not compute
 any Schnorr random real.
 By definition of $f_0,f_1$, $G$ is of $\DNR$ degree
 and is therefore weakly meager covering by Theorem \ref{schth0};
 and $G$ is of hyperimmune degree.
 Thus we are done.

\end{proof}

It is also clear that we can strengthen
Theorem 4.2 of \cite{khan2017forcing} as following.
\begin{corollary}
For any oracle $X$, there exists
a \DNR$^X$ that does not compute any Schnorr random real.
\end{corollary}

Given a countable collection
$\mcal{A}\subseteq 2^\omega$,
we say $X\in 2^\omega$ is \emph{$\mcal{A}$-cohesive}
if for every $A\in\mcal{A}$, either
$X\subseteq^* A$ or $X\subseteq^* \overline{A}$.
\cite{brendle2015analogy} section 5.3 asks whether there exists
a set $G$ being REC-cohesive such that $G$ does not compute
Schnorr random real.
We here gives a positive answer.
\begin{theorem}\label{schth5}
For any countable collection $\mcal{A}\subseteq 2^\omega$
with $A$ being $ \Delta_2^0$ for all $A\in \mcal{A}$,
there exists an infinite $\mcal{A}$-cohesive set $G$ such that
$G$ does not compute any Schnorr random real.
\end{theorem}
The proof of Theorem \ref{schth4} and \ref{schth5}
are similar. Both concern a modified version
of CJS style Seetapun foricing. One of the most well known
application of CJS method is
 to show that every $\Delta_2^0$ $k$-partition
of $\omega$ admit an infinite subset that
is low$_2$ \cite{Cholak2001strength}. Recently, Monin and Patey \cite{monin2019pigeons} carry out
a modified version of CJS style Seetapun foricing
and use it to prove a jump avoidance result.
Our construction can be seen as an effectivization
of their version.

\section{Proof of Theorem \ref{schth4} and \ref{schth5}}
\label{secproof}

Both proof are by constructing a sequence of condition
each forces a given requirement.
We firstly and mainly prove Theorem \ref{schth4},
and Theorem \ref{schth5}
is proved in exactly the same fashion.
As usual, a condition is seen as a collection of the candidates
of the weak solution we construct.
\begin{enumerate}
\item We firstly define condition, extension and forcing.
\item We establish some basic facts concerning forcing.
The key facts among them are: (a) for each $\Pi_2^0$
formula $\Phi$, a condition can be extended to force
$\Phi$ or $\neg\Phi$
(Lemma \ref{schlemforcing});
(b) forcing  a formula implies truth provided
 the set $G$ is constructed through a sufficiently generic filter
 (Lemma \ref{schlem7}).
 This part concerns a concept called valid
 (definition \ref{schdef3}).

\item Thirdly, we deal with requirements
concerning avoiding Schnorr randomness.
We show that if a condition forces a Turing functional
to be total, then it can be extended to force
a given requirement (Lemma \ref{schlem6}).

\end{enumerate}

Fix a $\ch k$-partition $\ch f:\omega\rightarrow \ch k$.

\begin{definition}
For $l\geq 1$, a class $\mcal{U}\subseteq 2^\omega$
is $l$-\emph{large}  if for every
$l$-partition $f$ of $\omega$, there exists an
$i<l$ such that $f^{-1}(i)\in \mcal{U}$. We say
$\mcal{U}$ is \emph{large} if it is $l$-large
for all $l\in\omega$.
\end{definition}
\begin{lemma}\label{schlem4}

If $\bigcap\limits_{e\in C}\mcal{U}_e$ is not large
for some countable set $C$,
then there  exists a finite subset $\t{C}$ of $C$
such that $\bigcap\limits_{e\in \t{C}}\mcal{U}_e$
is not large.

\end{lemma}
\begin{proof}

Let $f:\omega\rightarrow l$ be a partition of $\omega$ witnessing
 that $\bigcap\limits_{e\in C}\mcal{U}_e$ is not large. Suppose otherwise.
Let $C_n,n\in\omega$ be an increasing array of finite subset
of $C$ such that $\cup_n C_n=C$.
By the otherwise assumption, there must exist a $i<l$ such that for infinitely many $n$,
$f^{-1}(i)\in \bigcap\limits_{e\in C_n}\mcal{U}_e$.
But this means $f^{-1}(i)\in \bigcap\limits_{e\in C}\mcal{U}_e$,
a contradiction.

\end{proof}

From now on, for every $e\in\omega$, let $\mcal{U}_e\subseteq 2^\omega$ denote
the $e^{th}$ upward closed $\Sigma_1^0$ class.
The \emph{condition} we use is a tuple
$(\sigma^i_s, S, C: i<\ch k,s<r)$
such that
\begin{enumerate}
\item For each $s<r,i<\ch k$, $\sigma^i_s\subseteq \ch f^{-1}(i)$
for all $i<\ch k$; 

\item The set $S\subseteq r^{<\omega}\times \omega$ is a c.e. set
such that the projection of $S$ on $r^{<\omega}$,
namely $T_S$, is an infinite forest over a finite prefix free set $B$;
moreover for every $(\rho,l)\in S$, $l\geq |\rho|$
and $(\h\rho,l)\in S$ for all $\h\rho\preceq\rho\wedge\h\rho\in [B]^\preceq$;

\item The function $C: S\rightarrow \omega$ is computable
such that for every $(\rho,l)\in S$, $C(\rho,l)$ is seen as
the canonical index of a finite set; moreover,
$\bigcap_{e\in C(\rho,l)}\mcal{U}_e$ is $l$-large;

\item For every $\h l>l,\h\rho\succeq \rho$ with
$(\rho,l),(\h\rho,l),(\rho,\h l)\in S$, we have
$C(\h\rho,l)\supseteq C(\rho,l)\wedge
C(\rho,\h l)\supseteq C(\rho,l)$.

\end{enumerate}

\begin{remark}
In \cite{monin2019pigeons}, the role of $S$ is played by
a single $\Delta_2^0$ set controlling the jump of the
constructed solution. Due to the effectiveness of
Schnorr test, we will have to monitor the jump control more effectively.
As required by item (3), $S$ is such an effective way
to monitor the how the largeness  grow along path through  $T_S$.
The constructed solution $G$ will be a subset of some $Y\in [T_S]$.

Intuitively, each condition $d=(\sigma^i_s, S, C: i<\ch k,s<r)$
represents a collection of the candidates of the solution $G$ we construct, namely:
$\cup_{i<\ch k,r<s} [d^i_s]$ where
\begin{align}\nonumber
[d^i_s]= \big\{
\h G\in 2^{<\omega}\cup 2^\omega: \text{for some }Y\in [T_S],
\h G\succeq \sigma^i_s\wedge \h G\subseteq Y^{-1}(s)\cup \sigma^i_s
\big\}.
\end{align}

\end{remark}

A simple and intuitive observation is that, by Lemma \ref{schlem4}:
\begin{align}\label{scheq8}
\text{ for any $Y\in [T_S]$, $C_Y= \bigcap\limits_{(\rho,l)\in S, e\in C(\rho,l),\rho\preceq Y}\mcal{U}_e$
is large.}
\end{align}

A condition $\h d = (\tau^{ i}_{ s}, \h S, \h C  :   i<\ch k,s<\h r)$
\emph{extends} a condition $d = (\sigma^i_s, S, C:  i<\ch k,s<r)$
(written as $\h d\subseteq d$) if there exists a function $g: \h r\rightarrow r$,
a  $\t\rho$ with $[\t\rho]\cap [T_S]\ne\emptyset$  such that:
\begin{enumerate}
\item For every $i<\ch k, s<\h r$, $\tau^{i}_{ s}\succeq
\sigma^{i}_{g( s)}$;

\item For every $i<\ch k,s<\h r$,
$\tau^{i}_{s}\setminus \sigma^{i}_{g(s)}
\subseteq \t\rho^{-1}(g(s))$;

\item For every $(\h \rho,l)\in \h S$,
$\h\rho$ is a refinement of some element in
$T_S$. More precisely,
let $\rho$ be such that
  $ |\h{\rho}| = |\rho|$ and
 $\cup_{\h s:g(\h s) = s}\h{\rho}^{-1}(\h s)= \rho^{-1}(s)$ for all $s<r$,
 then $\rho\succeq\t\rho$, $(\rho,l)\in S$
 and $\h C(\h\rho,l)\supseteq C(\rho,l)$;

\end{enumerate}
In which case we say branch $\h s$ is a child branch of $g(\h s)$.
Intuitively, a condition $\h d$ extends $d$
means the collection $[\h d^i_{\h s}]$
 a sub collection of $[d^i_s]$
 if $\h s$ is a child branch of $s$.
It is easy but tedious to check that the extension relation is
transitive.

Given a formula $\Phi = \forall n\exists m\psi(G,n,m)$ where $\psi$
is $\Delta_0^0$, we let
$$\mcal{U}_{<\sigma,\Phi,n>} = \big\{
X: (\exists \tau\subseteq X\setminus |\sigma|)(\exists m)\
[\psi(\sigma\cup \tau,n,m)]\big\}.$$

\begin{definition}[Forcing]\label{schdef0}
Given a formula $\Phi$,
we define condition $d = (\sigma^i_s, S, C:  i<\ch k,s<r)$
\emph{forces } $\Phi$ on part $(s,i)$ (written as $d\vdash_{s,i} \Phi$)
as following: for some $\Delta_0^0$ formula $\psi$

\begin{enumerate}
\item When $\Phi = \psi(G)$,
then $d\vdash_{s,i}\Phi$
iff  $\psi(\sigma^i_s)$;

\item  When
$\Phi = \exists n \psi(G,n)$,
then $d\vdash_{s,i} \Phi$ iff $d\vdash_{s,i}\psi(G,n)$
for some $n$;

\item When $\Phi = \forall n \neg\psi(G,n)$,
then $d\vdash_{s,i} \Phi$ iff for every $n$, every
$\tau\in [d^i_s]$,
$\neg\psi(\tau,n)$;

\item When $\Phi = \exists n\forall m \neg\psi(G,n,m)$,
then $d\vdash_{s,i} \Phi$ iff for some $n$,
$d\vdash_{s,i} \forall m\neg\psi(G,n,m)$;

\item When $\Phi = \forall n\exists m \psi(G,n,m)$,
then $d\vdash_{s,i} \Phi$ iff
for every $\tau\in [d_s^i]$,
every $\rho\in T_S$, every $n$, if $  [\rho]\cap [T_S]\ne\emptyset$
and $\tau\setminus \sigma^i_s\subseteq \rho^{-1}(s)$, then
there exists an $l\geq n$ such that $(\rho,l)\in S$ and
$<\tau,\Phi,n>\in C(\rho,l)$.
\end{enumerate}

\end{definition}
Let $\h d=
(\tau^{i}_{s}, \h S, \h C: i<\ch k,s<\h r)\subseteq d=
 (\sigma^i_s, S, C: i<\ch k,s<r)$
 witnessed by $g,\t\rho$.
 Let $\Phi$ be a formula generated by
 a $\Delta_0^0$ formula $\psi$ as in one of the five
 items in definition \ref{schdef0}; moreover,
 $\psi$ satisfies
\begin{align}\nonumber
\text{ for every $n,m$, every $\tau'\succeq\tau$,
$\psi(\tau,n,m)\rightarrow \psi(\tau',n,m)$.
}
\end{align}
\begin{lemma}[Extension]
If $d\vdash_{s,i} \Phi$, then
for every child branch $ \h s $ of $s$,
$\h d\vdash_{\h s,i}\Phi$.

\end{lemma}
\begin{proof}

If $\Phi=\psi(G)$,
simply note that $\tau^i_{\h s}\succeq \sigma^i_s$,
therefore $\psi(\tau^i_{\h s})$ is true since $\psi(\sigma^i_s)$
is true.
If $\forall n\neg\psi(G,n)$,
since $[\h d^i_{\h s}]\subseteq [d^i_s]$,
therefore the conclusion follows.
The proof for $\Phi$ of form $\exists n\psi(G,n) ,\exists n\forall m\psi(G,n,m)$
follows similarly.

Suppose $\Phi= \forall n\exists m\psi(G,n,m)$.
Let $\tau\succeq \tau^{i}_{\h s}$ and $\h\rho$
satisfy $\tau\setminus \tau^{i}_{\h s}
\subseteq \h\rho^{-1}(\h s)$ 
 with $[\h\rho]\cap [T_{\h S}]\ne\emptyset$,
 let $n\in\omega$, we need to show that
there exists an $\h l\geq n$ such that $(\h\rho,\h l)\in \h S$
and $<\tau,\Phi,n>\in \h C(\h\rho,\h l)$.
Let $\h Y\in [\h \rho]\cap [T_{\h S}], Y\in [T_{S}]$
be such that $\h Y$ is refinement of $Y$ witnessed by $g$,
i.e., $\cup_{\t s:g(\t s)= s'}\h Y^{-1}(\t s) = Y^{-1}(s')$ for all $s'<r$.
Let $\rho = Y\uhr |\h\rho|$. Clearly
$\tau\setminus \sigma^i_s\subseteq \rho^{-1}(s)$. 
Since $d\vdash_{i,s}\Phi$,
by item (5) of forcing, there exists an
$l\geq n$ such that  $(\rho,l)\in S$
and $<\tau,\Phi,n>\in C(\rho,l)$.
Since $\h\rho\in T_{\h S}$ satisfy $[\h\rho]\cap [T_{\h S}]\ne\emptyset$,
there exists infinitely many $\h l$ such that $(\h\rho,\h l)\in \h S$
(see item (3) of the definition of condition).
Suppose $(\h \rho, \h l)\in \h S$ satisfy $\h l\geq l$.
But by item (3) of extension (and definition of
condition item (4)), $(\rho,\h l)\in S $ and
$\h C(\h\rho,\h l)\supseteq C(\rho,\h l)\supseteq C(\rho,l)\ni
<\tau,\Phi,n>$.
Thus we are done.

\end{proof}
\def\heteriditarillyvalid{heteriditarilly valid}

Due to the indirect nature of forcing item (5),
$d\vdash_{i,s}\Phi$ does necessarily implies
that $\Phi(G)$ is true for all $G\in [d^i_s]$.
Therefore we incorporate the notion of valid, which
roughly means, if branch $s$ is valid,
then whatever is forced is true provided $G$
is in that branch and constructed through a sequence
 of sufficiently generic conditions.
Because of the effectiveness of $S$, our definition of
validity is necessarily more tricky than that in \cite{monin2019pigeons}.

Let $d=
 (\sigma^i_s, S, C: i<\ch k,s<r)$ be a condition.
\begin{definition}[Valid]
\label{schdef3}
Given $s<r,\rho\in T_S$ with $[\rho]\cap [T_S]$.
\begin{itemize}
\item  We say
branch $s$ of $d$
 is \emph{valid over $\rho$ for $d$} if
 there exists a $Y\in[T_S]\cap [\rho]$
 such that for every $l$ with $(\rho,l)\in S$,
$Y^{-1}(s)\in \bigcap_{e\in C(\rho,l)}\mcal{U}_e$.
When $d$ is clear, we simply say branch $s$ is valid over $\rho$.
\item
We say branch $s$ of $d$ is \emph{valid in $d$} if it is valid
over every $\rho\in T_S$ such that $[\rho]\cap [T_S]\ne\emptyset$.
\item Condition $d$ is \emph{\heteriditarillyvalid} if
for every $s<r$, either it is valid or it is not valid
over any $\rho\in T_S$ with $[\rho]\cap [T_S]\ne\emptyset$.
\end{itemize}
\end{definition}

The next lemma says that roughly speaking,
for each condition $d$,
we can extends $d$ so that a valid branch   exist.

\begin{lemma}\label{schlem5}
Fix a condition $ d$
 and a branch $s$ of $d$.
\begin{enumerate}
\item
Condition $d$ admit a \heteriditarillyvalid\  extension.
\item
For  any $\rho$ with $[\rho]\cap [T_S]\ne\emptyset$,
 there exists an $s^*<r$ such that branch $s^*$ is valid over $\rho$.

 \item If $d$ is \heteriditarillyvalid, then there exists a branch $s^*$ such that
 $s^*$ is valid in $d$.

 \item If $\h d\subseteq d$ and  branch $s$ of $d$ is not valid over any $\rho$ with $[\rho]\cap [T_S]\ne\emptyset$,
 then for any child branch $\h s$ of $s$, $\h s$ is not valid
 over any $\h\rho$ with $[\h\rho]\cap [T_{\h S}]\ne\emptyset$.

\end{enumerate}
\end{lemma}
\begin{proof}
Proof of (1).
A simple observation is that if
$s$ is not valid over $\rho$ with $[\rho]\cap [T_S]\ne\emptyset$,
then $s$ is not valid over any $\rho'\in  T_S\cap [\rho]^\preceq$ with $[\rho']\cap [T_S]\ne\emptyset$.
Therefore there exists a $\rho^*\in T_S$ with $[\rho^*]\cap [T_S]\ne\emptyset$
such that for every $s<r$, either branch $s$ is not valid
over $\rho^*$, or $s$ is valid over every
$\rho'\in  T_S\cap [\rho^*]^\preceq$ with $[\rho']\cap [T_S]\ne\emptyset$.
Define $\h d$ as following,
let $\h S = \{(\rho,l)\in S:\rho\succeq \rho^*\}$
and the other component of $\h d$ is the same as $d$.
Clearly $\h d\subseteq d$ is \heteriditarillyvalid.

Proof of (2).
Fix $Y\in [\rho]\cap [T_S]$. As we have observed in (\ref{scheq8}),
$\bigcap_{e\in C_Y}\mcal{U}_e$ is large.
Therefore,    there must exist $s^*<r$
such that $Y^{-1}(s^*)\in \bigcap_{e\in C_Y}\mcal{U}_e$.
But $ \bigcap_{e\in C_Y}\mcal{U}_e
\subseteq \bigcap_{e\in C(\rho,l)}\mcal{U}_e$ for all $l$ with
$(\rho,l)\in S$. Thus we are done.

Item (3) is  direct  from item (2). Item (4) is also direct.
\end{proof}

\begin{definition}[Forcing question]
\label{schdef2}
Let $\Phi_i = \forall n\exists m\psi_i(G,n,m)$.
We say $d=
 (\sigma^i_s, S, C: i<\ch k,s<r)$
  \emph{potentially forces $\vee_{i< \ch k}\neg\Phi_i$
on branch $s$}
if there exists a
$Y\in [T_S]$, some set
$\mu^{i,0}_{s},
\cdots,\mu^{i,j_i}_s\subseteq (\ch f^{-1}(i)\cap Y^{-1}(s))\setminus |\sigma^i_s|$
for each $i<\ch k$ such that the class
$$\bigcap\limits_{
j\leq j_i,n\in\omega,i<\ch k}\mcal{U}_{<\sigma^{i}_s\cup \mu^{i,j}_s,\Phi_i,n>}\bigcap\limits_{e\in C_Y}\mcal{U}_e$$
is not large.
\end{definition}

The forcing question \ref{schdef2} is not quite the same
as \cite{monin2019pigeons}.
Especially that item (1) of \ref{schdef2} can be very complex to
decide since the given $\ch k$-partition $f$ can be arbitrary complex.
One of the applications in \cite{monin2019pigeons},
is to force  $\Psi^{G'}(n)\ne D(n)$
for some $n$ where $D$ is a given degree not computable in $\emptyset'$.
Therefore,
for each $n$, one need to enumerate the value of $\Psi^{G'}(n)$
by checking, for each $n$, the answer of the corresponding forcing question
(just like in the cone avoidance for $\Pi_1^0$ class,
where for each $n$, one need to enumerate the possible value of
$\Psi^G(n)$).
Thus they need the forcing question.
Here
 we do not need the effectiveness
of the forcing question.
One may wonder what about the effectiveness of Schnorr test.
When forcing $\Psi^G$ to be succeed by some Schnorr test,
we take advantage of $d_{s,i}$ forces “$\Psi^G$ is total",
so that the for each finite set of forcing question of particular form,
one of them admit a negative answer (i.e., the $\mu$ strings do not exist).
Thus overcome the effectiveness issue.

The key lemma is the following, which says that given a tuple of $\Pi_2^0$ formulas,
we can either force positive or negative of these formulas on a given branch.

\begin{lemma}\label{schlemforcing}
Fix an $s<r$, $k$ many formulas
 $\Phi_i = \forall n\exists m\psi_i(G,n,m)$.
\begin{enumerate}
\item If $d$ potentially force $\vee_i\neg\Phi_i$
on branch $s$, then there exists an extension $\h d$ of $d$
such that for every valid child branch $\h s$ of $s$, there exists $\h i<\ch k$
such that $\h d\vdash_{\h s,\h i}\neg\Phi_{\h i}$.

\item If $d$ does not potentially force $\vee_i\neg\Phi_i$
on branch $s$, then there exists an extension $\h d $ of $d$
such that for every child  branch $\h s$ of $s$ there exists an $\h i<\ch k$,
such that $d\vdash_{\h s,\h i}\Phi_{\h i}$.

\end{enumerate}

\end{lemma}
\begin{proof}
This is Lemma 3.10 of \cite{monin2019pigeons}.

Proof of (1).
Let $Y,\mu^{i,j}_s,j\leq j_i,i<\ch k$ be as in definition \ref{schdef2}.
By Lemma \ref{schlem4}, there must exist
a finite set $C\subseteq C_Y$, and an $n^*$
such that $$\bigcap\limits_{
j\leq j_i,i<\ch k,n<n^*}\mcal{U}_{<\sigma^{i}_s\cup \mu^{i,j}_s,\Phi_i,n>}\bigcap\limits_{e\in C}\mcal{U}_e$$
is not $n^*$-large.
Let $Q$ be the $\Pi_1^0$ class of $n^*$-partition of $\omega$
witnessing it to be not $n^*$-large and let
$X\in Q$ be $\Delta_2^0$.

We now split branch $s$ into $n^*$ many branches
and refine the mathias tail, namely members in $[T_S]$ by refining
them with $X$
as following.
 For each $n<n^*$,
 \begin{itemize}
 \item if
$X^{-1}(n)\notin \bigcap_{e\in C}\mcal{U}_e$,
then let $\tau^{i}_{(s,n)} = \sigma^i_s$ for all $i<\ch k$;

\item if $X^{-1}(n)\in \bigcap_{e\in C}\mcal{U}_e$,
then there must exist $i_n<\ch k, j\leq j_i$
such that $X^{-1}(n)\notin \bigcap_{n'<n^*}\mcal{U}_{<\sigma^{i_n}_s\cup \mu^{i_n,j}_s,\Phi_{i_n},n'>}$,
in which case let
$\tau^{i}_{(s,n)} = \sigma^i_s$ if $i\ne i_n$ and
$\tau^i_{(s,n)} = \sigma^i_s\cup \mu^{i_n,j}_s$ if $i=i_n$.
\end{itemize}
Note that since $\mu^{i,j}_s\cap |\sigma^i_s|=\emptyset$,
therefore $\tau^i_{(s,n)}\succeq \sigma^i_s$.
Now we define the other
component of the extension.
Since $X$ is $\Delta_2^0$, there exists
a c.e. tree $T$ (closed downward)  such that
$[T] = \{X\}$.
For $\rho\in T_S,\rho'\in T$ with $|\rho'|\geq \rho$,
let $(\rho,\rho')_s $ be such a string
 of length $|\rho|$ that  refines $\rho$ on part $s$,
 i.e.,
for every $m\leq |\rho|$
\begin{align}
(\rho,\rho')_s(m)  =
\left\{
\begin{aligned}
&\rho(m)\text{ if }\rho(m)\ne s;\\ \nonumber
&(s,n)\text{ if }\rho(m) = s\wedge \rho'(m) = n.
\end{aligned}
\right.
\end{align}
Let   $(\rho^*,l^*)\in S$ be such that
$\rho^*\prec Y$, $C(\rho^*,l^*)\supseteq C$
and $[\rho^*]\cap [T_S]\ne\emptyset$ (which must exist
since some initial segment of $Y$ can be $\rho^*$).
Let
\begin{align}\nonumber
&\h S = \big\{
((\rho,\rho')_s,l): (\rho,l)\in S, \rho'\in T, |\rho'| = |\rho|,
 \rho\succeq \rho^*, l\geq l^*
\big\} \text{ and }\\ \nonumber
&\h C(((\rho,\rho')_s,l)) = C(\rho,l)
\text{ for all }((\rho,\rho')_s,l)\in \h S
\end{align}

It's obvious that $\h d = (\tau^i_{s},\h S,\h C: i<\ch k, s<\h r)$
is a condition (especially checking item (3) the downward closeness of $T_S$)
extending $d$
witnessed by $\rho^*$
(especially checking item (3)). By Lemma \ref{schlem5}, we assume that $\h d$ is \heteriditarillyvalid.
Suppose branch $(s,n)$ is valid in $\h d$.

We show that for some $n'$ (depending on $n$),
 $\h d\vdash_{(s,n),i_n}\forall m\neg\psi_{i_n}(G,n',m)$.
Let $Z\in [T_{\h S}]$ and let $\tau\succeq \tau^{i_n}_{(s,n)}$ satisfy
$\tau\setminus \tau^{i_n}_{(s,n)}\subseteq Z^{-1}((s,n))\subseteq X^{-1}(n)$
(recall definition of $\h S$).
Since $\rho^*\prec Y$, therefore $$\h C_{Z}\supseteq C(\rho^*,l^*)\supseteq C.$$
Since $(s,n)$ is valid in $\h d$, there exists a $ \h Z\in [\rho^*]\cap [T_{\h S}]$
such that,
 $$\h Z^{-1}((s,n))\in \bigcap\limits_{e\in C(\rho^*,l^*)}\mcal{U}_e\subseteq
\bigcap\limits_{e\in C}\mcal{U}_e.$$
Therefore $X^{-1}(n)\in \bigcap_{e\in C}\mcal{U}_e$
since $\mcal{U}_e$ is closed upward and $\h Z^{-1}((s,n))\subseteq X^{-1}(n)$.
Since $X\in Q$  and by how we split branch $s$ (the second item),  for some $n'<n^*$,
$$X^{-1}(n)\notin
\mcal{U}_{<\tau^{i_n}_{(s,n)},\Phi_{i_n},n'>}.$$
Thus  by definition of $\mcal{U}_{<\tau^{i_n}_{(s,n)},\Phi_{i_n},n'>}$, we have
$\neg\psi_{i_n}(\tau,n',m)$ holds. Thus we are done.

\ \\

Proof of (2).
The branch $s$ is split into $k$ branches, namely
$(s,i),i<\ch k$.
We now define $\h S$ as following:
for each $(\rho,l)\in S$, wait for such a time
$t$ that for some $\rho'\in k^{<\omega}$ with $|\rho'| = |\rho|$,
 by the time $t$, it is found that the class
$$
\bigcap\limits_{e\in C(\rho,l)}\mcal{U}_e\cap \big(
\bigcap
\big\{\mcal{U}_{<\tau,\Phi_i,n>}:
i<\ch k, \tau\succeq \sigma^i_s,\tau\setminus\sigma^i_s\subseteq  \rho'^{-1}(i)\cap
\rho^{-1}(s), n\leq l\big\}\ \big)
$$
is $l$-large.
If such $t,\rho'$ exists for $(\rho,l)$,
then enumerate $((\rho,\rho')_s,l)$ into
$\h S$ and let
$$\h C(((\rho,\rho')_s,l)) = C(\rho,l)\cup
\big\{<\tau,\Phi_i,n>: i<\ch k, \tau\succeq \sigma^i_s,\tau\setminus\sigma^i_s\subseteq \rho'^{-1}(i)\cap
\rho^{-1}(s), n\leq l\big\}.$$
If such $t,\rho'$ does not exists for $(\rho,l)$, then do nothing.

We now verify the extension relation.
By definition of $\h S,\h C(\rho,l)$, it's easy to check that
item (3)   of extension is satisfied.
Moreover, by our hypothesis of this Lemma, $T_{\h S}$ must be infinite since
for every $(\rho,l)\in S$ with $[\rho]\cap [T_S]\ne\emptyset$,
such $\rho',t$ exists since an initial  segment of $\ch f$
could play the role of $\rho'$.
For every $i,\h i<\ch k, s<r$,
let $\tau^i_{(s,\h i)} = \sigma^i_{s}$.
The condition $\h d$ is the condition where initial segments of branch $s$
of $d$ is extended to $\tau^i_{(s,\h i)}$ and $S,C(\rho,l)$
are replaced by $\h S,\h C((\rho,\rho')_s,l)$ respectively.
It's trivial to verify other items of the definition
of extension.

We now verify forcing.
Fix a child branch $(s,\h i)$ of $s$,   we show that
$\h d\vdash_{(s,\h i),\h i}\Phi_{\h i}$.
Fix a $(\rho,\rho')_s$ with $[(\rho,\rho')_s]\cap[T_{\h S}]\ne\emptyset$,
a $\tau\succeq \tau^{\h i}_{(s,\h i)}$
with $\tau\setminus \tau^{\h i}_{(s,\h i)}\subseteq
\rho'^{-1}(\h i)\cap \rho^{-1}(s)$
and a $n\in\omega$, we need to show that for some $l\geq n$,
$((\rho,\rho')_s,l)\in \h S$ and
$<\tau,\Phi_{\h i},n>\in \h C((\rho,\rho')_s,l)$.
Note that $$\tau\succeq\sigma^{\h i}_s,
\tau\setminus \sigma^{\h i}_s\subseteq
\rho'^{-1}(\h i)\cap \rho^{-1}(s). $$
Since $[(\rho,\rho')_s]\cap[T_{\h S}]\ne\emptyset$,
there exists $l\geq n$ such that $((\rho,\rho')_s,l)\in \h S$.
By definition of $\h C((\rho,\rho')_s,l)$,
$<\tau,\Phi_{\h i},n>\in \h C((\rho,\rho')_s,l)$. Thus we are done.

\end{proof}

Now comes the combinatorics concerning Schnorr randomness.
For a Turing functional $\Psi$,
let $\Phi_\Psi=(\forall n)(\exists t )(\forall n'\leq n) \Psi^G(n')[t]\downarrow$;
we say $d\vdash_{s,i}\Psi\text{ is total }$ iff
$d\vdash_{s,i}\Phi_\Psi$;
for a finite set $V\subseteq 2^{<\omega}$,
let $\psi_\Psi(\sigma,m',t,V) =
(\Psi^\sigma\uhr m')[t]\downarrow\in [V]^\preceq$;
let $\mcal{U}_{\sigma,\Psi,V} =
\big\{
X\subseteq \omega: (\exists \rho\subseteq X\setminus|\sigma|)(\exists m)
(\Psi^{\sigma\cup\rho}\uhr m)\downarrow\notin [V]^\preceq
\big\}$;
for a Schnorr
test $\mb{V} = (V_0,V_1,\cdots)$, we
let $\psi_{\Psi,\mb{V}}(\sigma,n,m,m',t) =
m>n\wedge m'>n\wedge t>n\wedge \psi_{\Psi}(\sigma,m',t,V_m)$;
and let $\Phi_{\Psi,\mb{V}} = (\forall n)(\exists m,m',t) \psi_{\Psi,\mb{V}}(G,n,m,m',t)$.
Note that $\Phi_{\Psi,\mb{V}}(G)$
simply means that
\begin{align}\nonumber
\text{for every }n\in\omega\text{ there exists an }m>n
\text{ such that }\Psi^G\in [V_m].
\end{align}
 i.e., The test $\mb{V}$ succeeds
on $\Psi^G$.
\def\Q{Q}

\begin{lemma}\label{schlem6}
If $d\vdash_{s,i}\Phi_\Psi$, then there exists
a $\h d\leq d$, a Schnorr test $\mb{V}$ such that
$d\vdash_{\h s,i}\Phi_{\Psi,\mb{V}}$
for all child branch $\h s$ of $s$.
\end{lemma}
\begin{proof}

We firstly establish the following.
\begin{claim}\label{schclaim1}
For every $0<\lambda$, every $n$, there exists
$(\rho,l)\in S$ with $|\rho|\geq n$,
a $V\subseteq 2^{<
\omega}$ with $\m(V)\geq 1-\lambda$
such that
$$(\bigcap\limits_{e\in C(\rho,l)}\mcal{U}_e) \cap
\bigcap
\big\{ \mcal{U}_{\sigma,\Psi,V}:
\sigma\succeq\sigma^i_s
\wedge \sigma\setminus \sigma^i_s\subseteq
\rho^{-1}(s)\big\}$$
is $l$-large.

\end{claim}
\begin{proof}
For every finite set $V\subseteq 2^{<\omega}$, every $(\rho,l)\in S$,
consider the following $\Pi_1^0$ class $\Q_{\Psi,\rho,V,l}$ of $l$-partition of
$\omega$,
which roughly speaking forces $\Psi^G$ to be in $[V]$.
 More specifically, an $X\in l^\omega$ is in $\Q_{\Psi,\rho,V,l}$
iff  for every $l'<l$:
\begin{align}
 X^{-1}(l')\notin
(\bigcap\limits_{e\in C(\rho,l)}\mcal{U}_e)\cap
\bigcap
 \big\{\mcal{U}_{\sigma,\Psi,V}:
 \sigma\succeq\sigma^i_s
\wedge \sigma\setminus \sigma^i_s\subseteq
\rho^{-1}(s)
\big\}.
\end{align}
By definition of $\Q_{\Psi,\rho,V,l}$, for every  $(\rho,l)\in S$,
every $X\in \Q_{\Psi,\rho,V,l}$,
every $l'<l$, if $X^{-1}(l')\in
\bigcap_{e\in C(\rho,l)}\mcal{U}_e$,
then there exists $\sigma\succeq\sigma^i_s$
with $\sigma\setminus \sigma^i_s\subseteq
\rho^{-1}(s)$ 
such that for every $\tau\succeq\sigma$ with
$\tau\setminus\sigma\subseteq X^{-1}(l')$,
every $m>\max\{|\eta|:\eta\in V\}$,
\begin{align}\nonumber
\Psi^\tau\uhr m\downarrow\rightarrow
\Psi^\tau\uhr m\in [V]^\preceq.
\end{align}
Let $(\rho,l)\in S$ be such that $[\rho]\cap [T_S]\ne\emptyset
\wedge |\rho|\geq n\wedge l\geq N$
and
$<\sigma,\Phi_\Psi,N>\in C(\rho,l)$
for all $\sigma\succeq\sigma^i_s$
with $ \sigma\setminus \sigma^i_s\subseteq
\rho^{-1}(s)$ 
where
 $N$ is sufficiently large so that
 $2^{|\rho|}\cdot 2^{-N}\leq \lambda$.
By definition of forcing item (5),
such $(\rho,l)$ exists.

It suffices to show that
there exists a $V\subseteq 2^N$ with
$\m(V)\geq 1-\lambda$, such that
$\Q_{\Psi,\rho,V,l}=\emptyset$.
Suppose on the contrary that this is not the case,
we select a member from each $\Q_{\Psi,\rho,V,l}$
where $V$ traverse all subset of $2^N$ and show
that the refinement of these members together with
original condition forces the Turing functional
to be non total
since the output of that Turing functional must be
a common element of these $[V]$.

More precisely, suppose on the contrary, for each $V\subseteq 2^N$
with
$\m(V)\geq 1-\lambda$,
we have
$\Q_{\Psi,\rho,V,l}\ne\emptyset$.
Let $X_V\in \Q_{\Psi,\rho,V,l}$
and
\begin{align}
&X = (X_V: V\subseteq 2^N\wedge \m(V)\geq 1-\lambda),
\\ \nonumber
&\text{ i.e.,
$X$ is the refinement of all $X_V$.}
\end{align}
Since $[\rho]\cap [T_S]\ne\emptyset$,
$(\rho,l')\in S$ for infinitely many $l'$.
This means $\bigcap_{e\in C(\rho,l)}\mcal{U}_e$
is large. Therefore, there exists an
$l'$ such that $X^{-1}(l')\in \bigcap_{e\in C(\rho,l)}\mcal{U}_e$.
For each $V\subseteq 2^N$ with $\m(V)\geq 1-\lambda$,
 suppose $X^{-1}(l')\subseteq X_V^{-1}(l_V)$,
 since every $\mcal{U}_e$ is closed upward,
 we have $X^{-1}(l_V)\in \bigcap\limits_{e\in C(\rho,l)}\mcal{U}_e$.
This implies, by definition of
$\Q_{\Psi,\rho,V,l}$, that for every $V\subseteq 2^N$ with
$\m(V)\geq 1-\lambda$,
$$X_V^{-1}(l_V)\notin
\bigcap
 \big\{\mcal{U}_{\sigma,\Psi,V}:
 \sigma\succeq\sigma^i_s
\wedge \sigma\setminus \sigma^i_s\subseteq
\rho^{-1}(s)\big\}.$$
Therefore,  for every $V\subseteq 2^N$ with
$\m(V)\geq 1-\lambda$,
\begin{align}\label{scheq6}
X^{-1}(l')\notin\bigcap
 \big\{\mcal{U}_{\sigma,\Psi,V}
 : \sigma\succeq\sigma^i_s
\wedge \sigma\setminus \sigma^i_s\subseteq
\rho^{-1}(s)\big\}.
\end{align}
For every
 $\sigma\succeq\sigma^i_s$
with $ \sigma\setminus \sigma^i_s\subseteq
\rho^{-1}(s)$, 
since $<\sigma,\Phi_\Psi,N>\in C(\rho,l)$
and since $X^{-1}(l')\in \bigcap_{e\in C(\rho,l)}\mcal{U}_e$,
there exists a $\tau_\sigma\succeq\sigma$
with $\tau_\sigma\setminus \sigma\subseteq
X^{-1}(l')$ such that
$(\Psi^{\tau_\sigma}\uhr N)\downarrow$.
Let $$V^* = 2^N\setminus
\big\{\Psi^{\tau_\sigma}\uhr N:
\sigma\succeq\sigma^i_s
\wedge \sigma\setminus \sigma^i_s\subseteq
\rho^{-1}(s)\big\}. $$
Note that $\m(V^*)\geq 1-2^{|\rho|}\cdot 2^{-N}\geq
1-\lambda$ and by definition of $V^*$ and
$\mcal{U}_{\sigma,\Psi,V}$,
$$X^{-1}(l')\in\bigcap
 \big\{\mcal{U}_{\sigma,\Psi,V^*}:
 \sigma\succeq\sigma^i_s
\wedge \sigma\setminus \sigma^i_s\subseteq
\rho^{-1}(s)\big\},$$
witnessed by those $\tau_\sigma$.
This  contradicts with (\ref{scheq6}).

\end{proof}

Now we define the following c.e. set $\h S$ together with
a Schnorr test $\mb{V} = (V_0,V_1,\cdots)$ as following.
Suppose we have computed
$\h S[t] = \{(\rho_v,l_v): v< \h u\}$
and $V_m,m< u$.
Wait for the next time that it is found that
for some $N>u$, some
finite $V\subseteq 2^N$ with
$\m(V)\leq 4^{-u-1}$,
some $(\rho,l)\in S$ with
$|\rho|\geq \max\{l_v: v< \h u\}$,
the class
$$(\bigcap\limits_{e\in C(\rho,l)}\mcal{U}_e)\cap
\bigcap
 \big\{\mcal{U}_{\sigma,\Psi,2^N\setminus V}:
 \sigma\succeq\sigma^i_s
\wedge \sigma\setminus \sigma^i_s\subseteq
\rho^{-1}(s)\big\} $$ 
is $l$-large
(which exists by Claim \ref{schclaim1}).
Then for each $ \rho'\preceq\rho$ with $\rho'\in T_S$, enumerate
$(\rho',l)$ into $\h S$ (for which we say that $(\rho',l)$ is \emph{enumerated
into $\h S$ at step $\h u$ due to $(\rho,l)$}); and
let $V_u = V$.
Let $\mathbf{V} = (V_0,V_1,\cdots)$ as computed above.
For each $(\rho,l)\in \h S$,  enumerated into $\h S$ at step $u$,
define
$$\h C(\rho,l) = C(\rho,l)\cup
\big\{<\sigma,\Phi_{\Psi,\mb{V}},u'>:\sigma\succeq\sigma^i_s\wedge
 \sigma\setminus \sigma^i_s\subseteq
\rho^{-1}(s)
, u'<u\big\}.$$
By our construction of $\mathbf{V}$,
\begin{align}\nonumber
\bigcap\limits_{e\in \h C(\rho,l)}\mcal{U}_e
\text{ is $l$-large.}
\end{align}
Let $\h d= (\sigma^i_s,\h S,\h C: i<\ch k,s<r)$.
It's easy to verify that $\h d\subseteq  d$ (especially item (3) of definition of extension)
is a condition.
It remains to prove that $\h d\vdash_{s,i}\Phi_{\Psi,\mb{V}}$.
Fix a $\rho\in T_{\h S}$ with $[\rho]\cap [T_{\h S}]\ne\emptyset$,
a $\tau\succeq \sigma^i_s\wedge\tau\setminus\sigma^i_s\subseteq
\rho^{-1}(s)$ and an $n\in\omega$. 
We need to show that for some $l\geq n$,
$(\rho,l)\in \h S$ and $<\tau,\Phi_{\Psi,\mb{V}},n>\in \h C(\rho,l)$.
This follows by checking the definition of $\h S$ and those $\h C$ set.
More specifically, since $[\rho]\cap [T_S]\ne\emptyset$,
we have $(\rho,\h l)\in \h S$ for infinitely many $\h l$.
Suppose for some $ l\geq n$, $(\rho,l)$ is enumerated into $\h S$
due to $(\h\rho, l)$ at step $u$ with $u>n$.
By definition of $C( \rho, l)$,
  $<\tau,\Phi_{\Psi,\mb{V}},n>\in \h C(\rho,l)$.

\end{proof}

Let $d_0\geq d_1\geq\cdots$ be a sequence of condition.
We say $\{d_t\}_{t\in\omega}$ is \emph{$2$-generic} if for every
$\ch k$ many $\Pi_2^0$  formula $\Phi_i,i<\ch k$,
there exists a $t$ such that for every valid branch
$s$ of $d_t$, there exists a $i$ such that
$d_t\vdash_{s,i}\Phi_i\vee d_t\vdash_{s,i}\neg\Phi_i$.
By Lemma \ref{schlemforcing}, such $2$-generic sequence exists.
By Lemma \ref{schlem5} item (1), we may also assume that each $d_t$ is \heteriditarillyvalid.
By Lemma \ref{schlem5} item (4), the set of valid branches
of $d_t$ forms a finitely branching infinite tree $\mcal{T}$
(where the partial order is given by the child branch relation).
For convenience, we also assume that for every $t$, there is
a $n\in\omega$ such that for
every initial segment component $\sigma$ of $d_t$,
$|\sigma| = n$.
Let $(s_t:t\in\omega)$ be a path along $\mcal{T}$.
By paring argument, there exists a $i^*<\ch k$ such that
for every $\Pi_2^0$ formula $\Phi$,
there exists a $t$ such that $d_t\vdash_{s_t,i^*}\Phi\vee
d_t\vdash_{s_t,i^*}\neg\Phi$.
Let $G^* =\cup_t
\sigma^{i^*}_{s_t}$ which is well defined since
$\sigma^{i^*}_{s_{t+1}}\succeq \sigma^{i^*}_{s_t}$.
We need to show that forcing implies truth.
Let $\Phi(G) = \forall n\exists m\psi(G,n,m)$
where $\psi$ is such that
\begin{align}\label{scheq10}
\text{ for every $n,m$, every $\tau'\succeq\tau$,
$\psi(\tau,n,m)\rightarrow \psi(\tau',n,m)$.
}
\end{align}
Note that all formulas we concern about,
namely $\Phi_{\Psi}, \Phi_{\Psi,\mbV}$, the corresponding $\psi$ formula
satisfy (\ref{scheq10}).
\begin{lemma}[Truth]\label{schlem7}
If $d_t\vdash_{s_t,i^*}\Phi$
 $(d_t\vdash_{s_t,i^*}\neg\Phi$ respectively$)$ then
$\Phi(G^*)$ $(\neg\Phi(G^*)$ respectively$)$ is true.
\end{lemma}
\begin{proof}

This is Lemma 2.27 of \cite{monin2019pigeons}.

The proof for the case
$d_t\vdash_{s_t,i^*}\neg\Phi$ is simple.
Note that there exists
a $X\in [T_{S_t}]$ such that
$G^*\setminus \sigma^{i^*}_{s_t}\subseteq
X^{-1}(s_t)$. Thus the conclusion follows by definition of
forcing.

\ \\

Now we prove the case $d_t\vdash_{s_t,i^*}\Phi$.
Fix a $n$, we need to show that
$\exists m \psi(G^*,n,m)$.
Consider $\h\Phi = \exists m\forall \h n\psi(G,n,m)$
(yes $\h n$ does not actually appears in $\psi(G,n,m)$).
By $2$-generic
of $\{d_t\}_{t\in\omega}$ and definition of $i^*$, there exists
a $\h t\geq t$ such that
$$d_{\h t}\vdash_{s_{\h t},i^*}\h\Phi\vee
d_{\h t}\vdash_{s_{\h t},i^*}\neg\h \Phi.$$
If $d_{\h t}\vdash_{s_{\h t},i^*}\h\Phi$, which means
by definition of forcing item (4), for some $m$,
$d_{\h t}\vdash_{s_{\h t},i^*} \forall \h n\psi(G,n,m)$.
Thus
we are done by definition of forcing item (3).

Suppose $d_{\h t}\vdash_{s_{\h t},i^*}\neg\h\Phi$, i.e.,
$d_{\h t}\vdash_{s_{\h t},i^*}\forall m\exists \h n
\neg\psi(G,n,m)$.
Because $d_t\vdash_{s_t,i^*}\Phi$
(therefore $d_{\h t}\vdash_{s_{\h t},i^*}\Phi$),
 we have that
for some $(\rho,l)\in S_{\h t}$ with
$[\rho]\cap [T_{S_{\h t}}]\ne\emptyset$,
$ <\sigma^{i^*}_{s_{\h t}},\Phi,n>\in
C_{\h t}(\rho,l)$.
Since $s_{\h t}$ is valid in $d_{\h t}$,
there exists an $X\in [\rho]\cap [T_{S_{\h t}}]$,
such that $$X^{-1}(s_{\h t})\in \bigcap\limits_{e\in C_{\h t}(\rho,l)}
\mcal{U}_e\subseteq
 \mcal{U}_{<\sigma^{i^*}_{s_{\h t}},\Phi,n>}.$$
 Unfolding the definition of $ \mcal{U}_{<\sigma^{i^*}_{s_{\h t}},\Phi,n>}$,
 there exists a $\tau\succeq \sigma^{i^*}_{s_{\h t}}$ with
 $\tau\setminus \sigma^{i^*}_{s_{\h t}}\subseteq
 X^{-1}(s_{\h t})$ such that
 $\psi(\tau,n,m^*)$ for some $m^*$.
Since $d_{\h t}\vdash_{s_{\h t},i^*}\forall m\exists \h n
\neg\psi(G,n,m)$, suppose $\tau\setminus\sigma^{i^*}_{s_{\h t}}
\subseteq \h\rho^{-1}(s_{\h t})$ and $\h\rho\prec X$,
there exists a $(\h \rho,\h l)\in S_{\h t}$ with
$[\h\rho]\cap [T_{S_{\h t}}]\ne\emptyset$
such that $<\tau,\neg\h\Phi,m^*>\in C_{\h t}(\h\rho,\h l)$.
Since branch $s_{\h t}$ is valid in $d_{\h t}$,
there exists a $\h X\in [\h\rho]\cap [T_{S_{\h t}}]$
such that $$\h X^{-1}(s_{\h t})
\in \bigcap\limits_{e\in C_{\h t}(\h\rho,\h l)}
\mcal{U}_e\subseteq
 \mcal{U}_{<\tau,\neg\h\Phi,m^*>}. $$
 Unfolding the definition of $
 \mcal{U}_{<\tau,\neg\h\Phi,m^*>}$,
 there exists $\h\tau\succeq \tau$
 such that $\neg\psi(\h\tau,n,m^*)$,
 a contradiction with $\psi(\tau,n,m^*)$
since $\psi$ satisfies (\ref{scheq10}).

\end{proof}

Now we can prove Theorem \ref{schth4} and \ref{schth5}.
\begin{proof}
[Proof of Theorem \ref{schth4}]

Let $\Phi_p = \forall n\exists m [m>n\wedge m\in G]$.
Starting with
the following condition $d_0 = (\sigma^i_s,S_0,C_0: i<\ch k,s<r_0)$
where
$r_0 = 1$, $\sigma^i_s = \varepsilon$,
$S_0 = \{(\rho,l): \rho\in r_0^{<\omega},l\in\omega\}$,
$C_0(\rho,l) = \big\{<\tau,\Phi_p,n>: n\leq l\wedge \tau\subseteq \rho\big\}$.
By definition of $\Phi_p$, it's easy to see that
$d_0$ is a condition (especially the part
$\bigcap_{e\in C_0(\rho,l)}\mcal{U}_e$ is $l$-large).

Let $d_0\supseteq d_1\supseteq\cdots$ be a $2$-generic sequence as above
and additionally:

\begin{align} \label{scheq9}
&\text{ For every Turing functional $\Psi$, every $t$,
every branch $s$ of $d_t$ and every $i<\ch k$, }
\\ \nonumber
&\ \ \ \text{ if $d_t\vdash_{s,i}\Psi$ is total, }
\text{then
there exists a $\h t\geq t$, a Schnorr test
$\mb{V}$ such that }
\\ \nonumber
&\ \ \ \text{ for every
child branch $\h s$ of $s$,
$d_{\h t}\vdash_{\h s,i} \Phi_{\Psi,\mb{V}}$.}
\end{align}
This is possible by Lemma \ref{schlem6}.
Let $s_t,t\in\omega$ be a branch sequence such that
$s_{t+1}$ is a child of $s_t$ and
each $s_t$ is valid in $d_t$, let
$G^* = \cup_{t}
\sigma^{i^*}_{s_t}$.
Clearly by definition of condition item (1), $G^*\subseteq \ch f^{-1}(i^*)$.
By definition of
$d_0$, $d_0\vdash_{s_0,i^*}\Phi_p$,
therefore $G^*$ is infinite by Lemma \ref{schlem7}.
Moreover, by (\ref{scheq9}) and
 the construction of $\{d_t\}_{t\in\omega}$
 (also by Lemma \ref{schlem7}),
for every Turing functional $\Psi$, either
$\Psi^G$ is not total or there exists a  Schnorr test
$\mb{V}$ such that $\mb{V}$ succeeds on $\Psi^G$. Thus we are done.

\end{proof}
\begin{proof}[Proof of Theorem \ref{schth5}]

Let $d_0\geq d_1$ be as in the proof of Theorem \ref{schth4}
and additionally, for every $A\in \mcal{A}$,
there exists $t$ such that
for every $X\in [T_{S_t}]$, every $s<r_t$,
 $X^{-1}(s)\subseteq^*A\vee X^{-1}(s)\subseteq^* \overline{A}$
 (this is possible just like what we do in
 proof of Lemma \ref{schlemforcing} item (1)).
 Let  $G^*$  be as in the proof of Theorem \ref{schth4},
 we have that $G^*\subseteq f^{-1}(i^*)$ is infinite
 and does not compute a Schnorr random real; moreover, because of
 the additional requirement on $\{d_t\}_{t\in\omega}$,
 $G^*$ is $\mcal{A}$-cohesive. Thus we are done.

\end{proof}

The above proof and the
forcing we used depends heavily on the effectiveness of $S$ component
 of a condition.
 And to preserve the effectiveness
 of $S$-component, we can not realize an arbitrary partition
 by $S$-component. Therefore we do not know whether the following holds.

\begin{question}
Does every countable collection $\mcal{A}\subseteq 2^\omega$
admit an infinite $\mcal{A}$-cohesive set $G$ such that
$G$ does not compute any Schnorr random real.
\end{question}

\section{Weakness of Schnorr covering}
\label{secanotherques}
\def\mbV{\mb{V}}
\def\mST{\mcal{ST}}

An oracle $A$
\emph{Schnorr cover} a class $\mcal{A}$ if there
exists an $A$-Schnorr test $(V_n:n\in\omega)$
such that $\mcal{A}\subseteq \bigcap_n \bigcup_{m> n}V_m$.
A set $A\subseteq \omega$ is bi-immune
if neither $A$ or $\overline{A}$
contains an infinite computable set.
A Turing degree is bi-immune if it computes
a bi-immune set.
In the end of \cite{brendle2015analogy}, it is asked that whether
there exists
 a degree that is not bi-immune
 and Schnorr cover REC.

\begin{theorem}
For any countable class $\mcal{A}\subseteq 2^\omega$,
there exists a Turing degree $A$
such that $A$ Schnorr cover $\mcal{A}$
and $A$ is not bi-immune.
\end{theorem}
\begin{proof}
A \emph{test} is a sequence of finite set
$(V_n:n\in\omega)$ with $V_n\subseteq 2^{<\omega}$ such that $\m(V_n)\leq 4^{-n-1}$.
We construct a  test $\mbV^*$
such that $\mbV^*$ cover $\mcal{A}= \{A_s\}_{s\in\omega}$ and
$\mbV^*$, as an oracle, does not compute any bi-immune set.
Let $h:\omega\rightarrow\omega$ be an order function (computable and increasing).
In the following proof, we
restrict ourself to such  test
$\mbV=(V_0,V_1,\cdots)$ that
$V_n\subseteq 2^{h(n)}$.
We use $\v{V}$ to denote an initial segment of a test,
i.e., $\v{V}=(V_0,\cdots,V_n)$ for some $n$
and write $\v{V}(m)$ to denote the $m$-th component of $\v{V}$,
$|\v{V}|$ to denote the length of $\v{V}$.
For two initial segment of  test $\v{V}_0,\v{V}_1$,
we write $\v{V}_1\succeq\v{V}_0$ if
$\v{V}_0 = (V_0,\cdots,V_n)\wedge
\v{V}_1 = (V_0,\cdots,V_{n'})$ for some $n'\geq n$
similarly for notation $[\v{V}]^\preceq, [\v{V}]$.
We use bold face $\mbV$ to denote a  test
and let $\mST$ be the set of all initial segment of such
  test.
Note that in our setting, $\mST$ can be seen as
a computably bounded computable tree.

The \emph{condition} we use is a pair
$(\v{V},S)$ where $S\subseteq [\v{V}]^\preceq$ is a computable
tree of  test initial segment with no leaf
such that for every $\v{V}'\in S$ with
$|\v{V}'|_S $ being even, we have that
\begin{align}\label{scheq7}
\v{V}'^\smallfrown V\in S
\text{ for all }
V\text{ such that }
\v{V}'^\smallfrown V\in \mST.
\end{align}
Condition $(\v{V}',S')$ \emph{extends}
$(\v{V},S)$ (written as $(\v{V}',S')\leq
(\v{V},S)$ ) if $\v{V}'\in S\wedge S'\subseteq S$.
It is clear how (\ref{scheq7})  guarantee
that $\mbV^*$ covers $\mcal{A}$.
Fix a Turing functional and a condition
$(\v{V},S)$, it remains to show that we can extends
the condition to $(\v{V}^*,S^*)$
so that for every $\mbV\in [S^*]$,
$\Psi^\mbV$ is not bi-immune.

\textbf{Case 1.}
For every finitely many mutually incomparable
 $\v{V}_0,\cdots,\v{V}_{M-1}\in S$,
every $n$, there exists $n'>n$ and
$\v{V}'_m\in [\v{V}_m]^\preceq\cap S$ for each $m\leq M-1$ such that
$\Psi^{\v{V}'_m}(n')\downarrow=1$ for all $m\leq M-1$.

We inductively (and computably) define
a subset $S^*$ of $S$ together with a computable
set $A$ so that $A$ witness that $\Psi^\mbV$ is not
bi-immune for all $\mbV\in [S^*]$.
Suppose by time $t$ we have defined
$S^*$ up to level $2l+1$.
Let $\v{V}_0,\cdots,\v{V}_{M-1}$ be all elements
in $S^*$ at level $2l-1$.
Note that by hypothesis of Case 1, there exists
a $n'> A[t] $,
$\v{V}_m'\in [\v{V}_m]^\preceq\cap S$ for each $m\leq M-1$
such that $\Psi^{\v{V}'_m}(n')\downarrow = 1$ for all
$m\leq M-1$. Clearly
such $n'$ and $\v{V}_m'$ can be computed.
Moreover, clearly we may assume without loss of generality
that $|\v{V}'_m|_S$ is even for all $m\leq M-1$
(otherwise extend them to be so) and they are mutually incomparable.
The $2l$ level of $S^*$ consists of $\v{V}'_m,m\leq M-1$,
the $2l+1 $ level of $S^*$ consists of
$\v{V}_m'^\smallfrown V$ for all $m\leq M-1$ and $V$
such that $\v{V}_m'^\smallfrown V\in S$.
Then we  enumerate $n'$ into $A$.
It is easy to check that $(\v{V},S^*)$ is  the desire extension

\textbf{Case 2}. Otherwise.

Suppose $\v{V}_0,\cdots,\v{V}_{M-1}\in S, n\in\omega$ witness
the otherwise hypothesis, i.e.,
there exists no $n'>n$ and
$\v{V}'_m\in [\v{V}_m]^\preceq\cap S$ for each $m\leq M-1$ such that
$\Psi^{\v{V}'_m}(n')\downarrow=1$ for all $m\leq M-1$.
Moreover, suppose $\v{V}_0,\cdots,\v{V}_{M-1}$ is minimal
in the sense that no (actual) subset of
 $\v{V}_0,\cdots,\v{V}_{M-1}$ can be a witness.
 Note that if $M=1$, then the hypothesis of Case 2 means
 that for every $\mbV\in [\v{V}_0]^\preceq\cap S$,
 $\Psi^{\mbV}\subseteq \{0,\cdots,n\}$ if it is total.
 Thus let $\v{V}^*=\v{V}_0$ and
 let $S^*\subseteq [\v{V}_0]^\preceq\cap S$ be a computable
 tree so that $(\v{V}^*,S^*)$ is a condition, then
 it is clear that this condition forces $\Psi^G$
 to be finite.
 If $M>1$, which means
 $\v{V}_1,\cdots,\v{V}_{M-1}$ is not a witness for the otherwise hypothesis,
 then as in Case 1, we can compute an infinite
 set $A$ such that for every $n'\in A$,
 there exists $\v{V}'_m\in [\v{V}_m]^\preceq\cap S$
 for each $1\leq m\leq M-1$ such that
 $\Psi^{\v{V}'_m}(n')\downarrow= 1$. This means that
 for every $n'\in A$ and every $\v{V}'\in [\v{V}_0]^\preceq\cap S$,
 $\Psi^{\v{V}'}(n')\uparrow\vee \Psi^{\v{V}'}(n')=0$.
 Thus let $\v{V}^* = \v{V}_0$ and let
 $S^*\subseteq [\v{V}_0]^\preceq\cap S$ be a computable
 tree so that $(\v{V}^*,S^*)$ is a condition, then
 it is clear that for every $\mbV\in [S^*]$,
 $\Psi^\mbV\cap A=\emptyset$ if it is total. Thus we are done.

\end{proof}
\bibliographystyle{amsplain}
\bibliography{F:/6+1/Draft/bibliographylogic}

\end{document}